\documentclass[11pt,twoside]{article} 
\usepackage{bbm}
\usepackage{amsmath,amssymb}
\usepackage{slashed}
\usepackage{float}
\usepackage{esint}%
\usepackage{tikz}
\usepackage{hyperref}
\usepackage{esint}
\makeatletter
\newcommand*{\mint}[1]{%
  \mint@l{#1}{}%
}
\newcommand*{\mint@l}[2]{%
  \@ifnextchar\limits{%
    \mint@l{#1}%
  }{%
    \@ifnextchar\nolimits{%
      \mint@l{#1}%
    }{%
      \@ifnextchar\displaylimits{%
        \mint@l{#1}%
      }{%
        \mint@s{#2}{#1}%
      }%
    }%
  }%
}
\newcommand*{\mint@s}[2]{%
  \@ifnextchar_{%
    \mint@sub{#1}{#2}%
  }{%
    \@ifnextchar^{%
      \mint@sup{#1}{#2}%
    }{%
      \mint@{#1}{#2}{}{}%
    }%
  }%
}
\def\mint@sub#1#2_#3{%
  \@ifnextchar^{%
    \mint@sub@sup{#1}{#2}{#3}%
  }{%
    \mint@{#1}{#2}{#3}{}%
  }%
}
\def\mint@sup#1#2^#3{%
  \@ifnextchar_{%
    \mint@sup@sub{#1}{#2}{#3}%
  }{%
    \mint@{#1}{#2}{}{#3}%
  }%
}
\def\mint@sub@sup#1#2#3^#4{%
  \mint@{#1}{#2}{#3}{#4}%
}
\def\mint@sup@sub#1#2#3_#4{%
  \mint@{#1}{#2}{#4}{#3}%
}
\newcommand*{\mint@}[4]{%
  \mathop{}%
  \mkern-\thinmuskip
  \mathchoice{%
    \mint@@{#1}{#2}{#3}{#4}%
        \displaystyle\textstyle\scriptstyle
  }{%
    \mint@@{#1}{#2}{#3}{#4}%
        \textstyle\scriptstyle\scriptstyle
  }{%
    \mint@@{#1}{#2}{#3}{#4}%
        \scriptstyle\scriptscriptstyle\scriptscriptstyle
  }{%
    \mint@@{#1}{#2}{#3}{#4}%
        \scriptscriptstyle\scriptscriptstyle\scriptscriptstyle
  }%
  \mkern-\thinmuskip
  \int#1%
  \ifx\\#3\\\else_{#3}\fi
  \ifx\\#4\\\else^{#4}\fi
}
\newcommand*{\mint@@}[7]{%
  \begingroup
    \sbox0{$#5\int\m@th$}%
    \sbox2{$#5\int_{}\m@th$}%
    \dimen2=\wd0 %
    \let\mint@limits=#1\relax
    \ifx\mint@limits\relax
      \sbox4{$#5\int_{\kern1sp}^{\kern1sp}\m@th$}%
      \ifdim\wd4>\wd2 %
        \let\mint@limits=\nolimits
      \else
        \let\mint@limits=\limits
      \fi
    \fi
    \ifx\mint@limits\displaylimits
      \ifx#5\displaystyle
        \let\mint@limits=\limits
      \fi
    \fi
    \ifx\mint@limits\limits
      \sbox0{$#7#3\m@th$}%
      \sbox2{$#7#4\m@th$}%
      \ifdim\wd0>\dimen2 %
        \dimen2=\wd0 %
      \fi
      \ifdim\wd2>\dimen2 %
        \dimen2=\wd2 %
      \fi
    \fi
    \rlap{%
      $#5%
        \vcenter{%
          \hbox to\dimen2{%
            \hss
            $#6{#2}\m@th$%
            \hss
          }%
        }%
      $%
    }%
  \endgroup
}
\usepackage{mathrsfs}
\usepackage{amssymb}
\usepackage{amsmath}
\usepackage{amsthm}
\usepackage{amsfonts}
\usepackage{color}
\usepackage{graphicx}
\usepackage[active]{srcltx}
\usepackage{tikz}
\usepackage{pgflibraryarrows}
\usepackage{pgflibrarysnakes}
\usepackage{extarrows}

\usepackage[cp1252]{inputenc}

\usepackage{mathrsfs}
\usepackage{graphicx}

\usepackage[active]{srcltx}

\allowdisplaybreaks

\usepackage{titletoc}
\titlecontents{section}[0pt]{\addvspace{2pt}\filright}
{\contentspush{\thecontentslabel\ }}
              {}{\titlerule*[8pt]{.}\contentspage}

\def\bint{{\ifinner\rlap{\bf\kern.35em--}
\int\else\rlap{\bf\kern.45em--}\int\fi}\ignorespaces}

\def\bbint{{\ifinner\rlap{\bf\kern.35em--}
\hspace{0.078cm}\int\else\rlap{\bf\kern.45em--}\int\fi}\ignorespaces}

\newtheorem{thm}{Theorem}[section]
\newtheorem{lem}[thm]{Lemma}
\newtheorem{rem}[thm]{Remark}
\newtheorem{cor}[thm]{Corollary}
\numberwithin{equation}{section}

\title
{\Large\bf  Homeomorphic Extensions in Bi--Orlicz--Sobolev Spaces
\footnotetext{\hspace{-0.35cm}
\endgraf
The author is supported by the Academy of Finland, project no. 334014.\\
Email: yizhu@jyu.fi\\
{\bf Key words and phrases:} 
Orlicz-Sobolev homeomorphisms, Bi-Orlicz-Sobolev extensions, $\Phi$-Douglas condition.\\
{\bf 2020 MSC:} 46E35
 \endgraf 
}}
\author{Yizhe Zhu}
 \date{}
\begin{document}

\arraycolsep=1pt
\allowdisplaybreaks
 \maketitle
 \begin{abstract}
We provide a complete characterization of those self-homeomorphisms of the unit circle that admit homeomorphic extensions to the unit disk belonging to bi--Orlicz--Sobolev spaces. Our results generalize classical criteria from the Sobolev setting to the more flexible Orlicz framework.
\end{abstract}

\section{Introduction}
 By the classical theory of Rad\'o~\cite{rado1926aufgabe}, Kneser~\cite{kneser1926losung} and Choquet~\cite{choquet1945type}, see also~\cite{duren2004harmonic}, every boundary homeomorphism $\varphi \colon \partial \mathbb{D}\xrightarrow{\rm onto}\partial \mathbb{D}$ admits a smooth harmonic diffeomorphic extension of the unit disk $\mathbb{D} \subset \mathbb C$ onto itself. Moreover,
since  the harmonic extension of $\varphi$  has the smallest Dirichlet energy among all extensions, a homeomorphism $\varphi \colon \partial \mathbb D \xrightarrow{\rm onto} \partial \mathbb D$ admits a homeomorphic extension $h \colon \overline{\mathbb{D}} \xrightarrow{\rm onto} \overline{\mathbb{D}}$ in $W^{1,2} (\mathbb{D}, \mathbb{C})$ if and only if $\varphi$ satisfies the Douglas condition:
\begin{equation}\label{eq:bidouglas}
\int_{\partial \mathbb D} \int_{\partial \mathbb D}  \left| \frac{\varphi (\xi) - \varphi (\eta)}{ \xi - \eta }\right|^2|d\xi||d\eta| < \infty \ . 
\end{equation}
See also \cite{koski2021sobolev,koski2023sobolev} for this this characterization. More generally, for $p>1$, the $p$-harmonic variants of the Rad\'o-Kneser-Choquet theorem~\cite{alessandrini2001geometric}  show that if  a homeomorphism $\varphi \colon \partial \mathbb D \xrightarrow{\rm onto} \partial \mathbb D$  satisfies the $p$-Douglas condition
\begin{equation}\label{eq:p-douglas}
\int_{\partial \mathbb D} \int_{\partial \mathbb D}  \left| \frac{\varphi (\xi) - \varphi (\eta)}{ \xi - \eta }\right|^p|d\xi||d\eta| < \infty \, ,
\end{equation}
then it actually admits  a homeomorphic extension $h \colon \overline{\mathbb{D}} \xrightarrow{\rm onto} \overline{\mathbb{D}}$ in $W^{1,p} (\mathbb{D},\mathbb{C})$; see also \cite{iwaniec2019rado,koski2018radial,iwaniec2014rado} for this this characterization.
Interestingly, this extension may also be obtained via the standard harmonic extension operator~\cite{koski2023bi}. It is also known  that the harmonic extension of an arbitrary homeomorphism $\varphi \colon \partial \mathbb{D} \xrightarrow{\rm onto} \partial \mathbb{D}$ lies in $W^{1,p}(\mathbb{D}, \mathbb{C})$ for every $p \in [1,2)$; see~\cite{verchota2007harmonic}.

Despite its robustness, the harmonic extension operator provides poor simultaneous control of the Sobolev regularity of both a map and its inverse. Indeed, there exists a homeomorphism $\varphi \colon \partial \mathbb{D} \xrightarrow{\rm onto} \partial \mathbb{D}$ satisfying the bi-Douglas condition, that is, both $\varphi$ and $\varphi^{-1}$ satisfy~\eqref{eq:bidouglas} such that the inverse of its harmonic extension does not belong to $W^{1,2}(\mathbb{D}, \mathbb{C})$; see \cite{koski2023bi}.

In recent work \cite{koski2023bi} by Onninen and Koski, new flexible techniques have been developed to control the Sobolev norms of both a mapping and its inverse. Specifically, if $\varphi \colon \partial \mathbb{D} \xrightarrow{\rm onto} \partial \mathbb{D}$ satisfies the $p$-Douglas condition and its inverse $\varphi^{-1}$ satisfies the $q$-Douglas condition, then there exists a homeomorphic extension $h \colon \mathbb{D} \xrightarrow{\rm onto} \mathbb{D}$ such that $h \in W^{1,p}(\mathbb{D}, \mathbb{D})$ and $h^{-1} \in W^{1,q}(\mathbb{D}, \mathbb{D})$.

In this paper, we extend this framework to Orlicz-Sobolev spaces of the form
\[
W^{1,\Phi}(\mathbb{D}, \mathbb{C}): = \left\{ f \in W^{1,1}(\mathbb{D}, \mathbb{C}) : \int_{\mathbb{D}} \Phi(|Df(x)|) \, dx < \infty \right\},
\]
where $\Phi \colon [0, \infty) \to [0, \infty)$ is an $N$-function satisfying the $\Delta_2$-condition and ${\rm (aInc)}_p$ on $t\ge t_0$ for some $p>1$.

A function $\Phi$ is called an \emph{$N$-function} if it is continuous, convex, increasing, with $\Phi(0) = 0$, and
\begin{equation}\label{eq:N-function}
\lim_{t \to 0^+} \frac{\Phi(t)}{t} = 0 \quad \text{and} \quad \lim_{t \to \infty} \frac{\Phi(t)}{t} = \infty.
\end{equation}

An $N$-function $\Phi$ can be expressed as
\begin{align}\label{N-function}
    \Phi(t)=\int_0^t \phi(s)ds,
\end{align}
where $\phi:[0,\infty)\to[0,\infty)$ is an increasing, right-continuous function with $\phi(0)=0$ and $\lim\limits_{t\to\infty}
\phi (t)=\infty$; see \cite{musielak2006orlicz,adams2003sobolev} for further details on $N$-functions. 
We say that $\Phi$ satisfies the \emph{doubling condition} (or $\Delta_2$-condition) if there exists a constant $C_{D}\ge 2$ such that
\begin{equation} \label{eq:doubling}
\Phi(2t) \leq C_{D} \Phi(t) \quad \text{for all } t \geq 0.
\end{equation}
Moreover, we say that $\Phi$ satisfies ${\rm (aInc)}_p$ if $\Phi(t)/t^p$ is {\em almost increasing}, that is, there exists a constant $C_{A}\ge 1$ such that
\begin{equation}\label{ainp}
\frac{\Phi(s)}{s^p}\le C_{A}\frac{\Phi(t)}{t^p} \quad {\rm whenever} \ 0\le s\le t.
\end{equation}

 We define the \emph{$\Phi$-Douglas condition} for a homeomorphism $\varphi \colon \partial \mathbb{D} \to \mathbb{C}$ by
\begin{equation}\label{eq:phi-douglas}
\int_{\partial \mathbb{D}} \int_{\partial \mathbb{D}} \Phi\left( \left| \frac{\varphi(\xi) - \varphi(\eta)}{\xi - \eta} \right| \right) \, |d\xi|\,|d\eta| < \infty.
\end{equation}
We are now in a position to state our main result.
\begin{thm}\label{thm:rkcp>2}
Let $\varphi \colon \partial \mathbb{D} \xrightarrow{\rm onto} \partial \mathbb{D}$ be a homeomorphism. Suppose $\Phi, \Psi \colon [0, \infty) \to [0, \infty)$ are $N$-functions that satisfy the doubling condition and ${\rm (aInc)}_p$ on $t\ge t_0$ for some $p>1$. Then $\varphi$ satisfies the $\Phi$-Douglas condition and $\varphi^{-1}$ satisfies the $\Psi$-Douglas condition if and only if $\varphi$ admits a homeomorphic extension $h \colon \overline{\mathbb{D}} \xrightarrow{\rm onto} \overline{\mathbb{D}}$ such that $h \in W^{1,\Phi}(\mathbb{D}, \mathbb{D})$ and $h^{-1} \in W^{1,\Psi}(\mathbb{D}, \mathbb{D})$.
\end{thm}
\begin{rem}\label{rem only if part}
    Concerning Theorem \ref{thm:rkcp>2}, we remark that the condition ${\rm (aInc)}_p$ is not needed for the proof of the ``only if'' part. 
\end{rem}
Every bounded homeomorphism $\varphi:\partial\mathbb{D}\xrightarrow{ \rm onto}\partial\mathbb{D}$ satisfies the $p$-Douglas condition, for $p\in(1,2)$. In fact, we prove a stronger statement: every such $\varphi$ also satisfies the $\Phi$-Douglas condition whenever $\int_1^\infty\frac{\Phi(t)}{t^3}dt<\infty$. Combining this result with our main theorem yields the following corollary.

\begin{cor}\label{cor: phi}
Let $\Phi$ be an $N$-function satisfying the doubling condition and
\begin{equation} \label{eq:p<2orlicz}
\int_1^{\infty}\frac{\Phi(t)}{t^3}dt<\infty.
\end{equation}
Then every homeomorphism $\varphi \colon \partial \mathbb{D} \xrightarrow{\rm onto} \partial \mathbb{D}$ admits a homeomorphic extension $h \colon \overline{\mathbb{D}} \xrightarrow{\rm onto}  \overline{\mathbb{D}}$ such that both $h$ and $h^{-1}$ belong to the Orlicz--Sobolev space $W^{1,\Phi}(\mathbb{D}, \mathbb{C})$.
\end{cor}

In particular, we characterize bounded self-homeomorphisms of the unit circle that admit an Orlicz-Sobolev homeomorphic extension.
 
\begin{cor}\label{cor: direct}
Let $\varphi \colon \partial \mathbb{D} \xrightarrow{\rm onto}\partial \mathbb{D}$ be a homeomorphism. Suppose $\Phi \colon [0, \infty) \to [0, \infty)$ is an $N$-function that satisfies the doubling condition and ${\rm (aInc)}_p$, for some $p>1$ and $t_0\ge0$, when $t\ge t_0$. Then $\varphi$ satisfies the $\Phi$-Douglas condition if and only if it admits a homeomorphic extension in $ W^{1,\Phi}(\mathbb{D}, \mathbb{C})$.
\end{cor}

\noindent{\bf Acknowledgments.} The author thanks his thesis advisor Jani Onninen for his ideas and all his help in improving the manuscript. The author also thanks the referee(s) of this paper for multiple suggested improvements to the presentation of the paper.
\section{Preliminaries}
In this section, we recall some lemmas that are needed in the proofs to follow.
\subsection{{\em N}-functions}
\begin{lem}\label{inverse convex}
For every $N$-function $\Phi:[0,\infty)\to[0,\infty)$ and $0\le a, b<\infty$, we have 
\begin{align*}
    \Phi(a)+\Phi(b)\le\Phi(a+b).
\end{align*}
\end{lem}
\begin{proof}
Since $\phi$ is an increasing function and 
\begin{equation*}
    \int_a^{a+b}\phi(s)ds-\int_0^a\phi(s)ds=\int_0^a\phi(s+b)-\phi(s)ds\ge0.
\end{equation*}
Thus we obtain
\begin{align*}
\Phi(a+b)-\Phi(b)=\int_b^{a+b}\phi(s)ds\ge\int_0^a\phi(s)ds=\Phi(a).
\end{align*}
\end{proof}
For each $N$-function $\Phi$, we define the Orlicz-Lebesgue space of the form
\begin{equation*}
    L^\Phi(\Omega,\mathbb{C}):=\left\{f: \int_\Omega\Phi(|f(x)|)dx<\infty\right\}.
\end{equation*}
If $\Phi$ additionally satisfies ${\rm (aInc)}_p$ on $t\ge t_0$ for some $p>1$, then  the maximal operator $M$ is bounded on $L^\Phi(\Omega,\mathbb{C})$, that is, $M:L^\Phi(\Omega,\mathbb{C})\hookrightarrow L^\Phi(\Omega,\mathbb{C})$.

\begin{lem}\label{maximum embedding} 
Let $\Phi:[0,\infty]\to[0,\infty]$ be an $N$-function satisfying the doubling condition and ${\rm(aInc)}_p$ on $t\ge t_0$ for some $p>1$. 
Let $M$ denote the Hardy-Littlewood maximal operator, and let $\Omega $ be a domain. Then for every $f\in L^\Phi(\Omega,\mathbb{C})$,
we have 
\begin{equation*}
\int_\Omega\Phi(Mf(x))dx\le C\int_{\Omega}\Phi(|f(x)|)dx<\infty+|\Omega|\Phi(t_0),
\end{equation*}
where $C$ only depends on $p$ and $C_A$.
\end{lem}
\begin{proof}
Throughout the proof, all implicit constants in estimates of the form $X\lesssim Y$ only depend on $p$ and $C_A$. We begin the proof by decomposing the estimate into two parts,
\begin{equation*}
\begin{aligned}
\int_{\Omega}\Phi(Mf(x))dx
&=\int_\Omega\int_0^{Mf(x)}
\phi(s)dsdx\\
&=\int_0^{t_0}\phi(s)|\{ Mf(x)>s\}|ds+\int_{t_0}^{\infty}\phi(s)|\{Mf(x)>s\}|ds\\
&=:I_1+I_2.
\end{aligned}
\end{equation*}
The first integral $I_1$ is easy to estimate. Indeed,
\begin{equation}\label{I1}
    I_1\le |\Omega|\Phi(t_0).
\end{equation}
By the Hardy-Littlewood maximal weak (1,1) estimate, 
\begin{equation*}
|\{Mf(x)>s\}|\lesssim \frac{1}{s}\int_{\{|f|>s\}}|f|dx
\end{equation*}
and integration by parts, we have
\begin{equation}\label{I2}
\begin{aligned}
I_2&\lesssim \int_{t_0}^\infty \phi(s)s^{-1}\int_{\{|f|>s\}}|f|dxds\\
&\le \int_\Omega|f|\int_{t_0}^{|f|}s^{-1}\phi(s)dsdx\\
&=\int_{\Omega}|f|\left[\frac{\Phi(|f|)}{|f|}-\frac{\Phi(t_0)}{t_0}+\int_{t_0}^{|f|}s^{-2}\Phi(s)ds\right]dx\\
&\le \int_{\Omega}\Phi(|f|)dx+\int_{\Omega}|f|\int_{t_0}^{|f|}s^{-2}\Phi(s)dsdx.
\end{aligned}
\end{equation}
Since $\Phi$ is ${\rm(aInc)}_p$ on $t\ge t_0$ for some $p>1$, we obtain
\begin{equation}\label{I2plus}
\begin{aligned}
\int_\Omega|f|\int_{t_0}^{|f|}s^{-2}\Phi(s)dsdx\lesssim \int_{\Omega}|f|^{1-p}\Phi(|f|)\int_{t_0}^{|f|}s^{p-2}dsdx\le\int_\Omega\Phi(|f|)dx. 
\end{aligned}
\end{equation}
Combining \eqref{I1}, \eqref{I2} and \eqref{I2plus}, we conclude that 
\begin{equation*}
\int_\Omega\Phi(Mf(x))dx\le\int_\Omega\Phi(|f|)dx+ |\Omega|\Phi(t_0).
\end{equation*}
\end{proof}

\subsection{Alternative to $\Phi$-Douglas condition}
For the proofs to be presented in the next sections, we give an alternative formulation of the $\Phi$-Douglas condition \eqref{eq:phi-douglas}. To state this, we let $I_{n,k}$ denote a dyadic decomposition of the unit circle. Then the result reads as follows.
\begin{lem}\label{alternative formulation lemma}
    The boundary homeomorphism $\varphi:\partial\mathbb{D}\xrightarrow{\rm onto}\partial\mathbb{D}$ satisfies the so-called $\Phi$-Douglas condition if and only if 
\begin{align}\label{alternative formulation}
\sum_{n=1}^\infty\sum_{k=1}^{2^n}\Phi(|\varphi(I_{n,k})|2^n)2^{-2n}<\infty.
\end{align}
\end{lem}


\begin{proof} We first divide the double integral on $\partial\mathbb{D}$ into two parts,
\begin{align*}
\int_{\partial\mathbb{D}}\int_{\partial\mathbb{D}}\Phi\left(\frac{|\varphi(x)-\varphi(y)|}{|x-y|}\right)|dx||dy|&=\int_{\partial\mathbb{D}}\int_{|x-y|\le 1}\Phi\left(\frac{|\varphi(x)-\varphi(y)|}{|x-y|}\right)|dx||dy|\\
&+\int_{\partial\mathbb{D}}\int_{|x-y|> 1}\Phi\left(\frac{|\varphi(x)-\varphi(y)|}{|x-y|}\right)|dx||dy|.
\end{align*}
Since the second integral on the right side of the above equality is always finite, we replace the so-called $\Phi$-Douglas condition with the finiteness of the first integral, which can be written as
\begin{align*}
 \int_{\partial\mathbb{D}}\int_{|x-y|\le 1}\Phi\left(\frac{|\varphi(x)-\varphi(y)|}{|x-y|}\right)|dx||dy|=\int_{\partial\mathbb{D}}\sum_{n=1}^\infty\int_{2^{-n}<|x-y|\le 2^{-(n-1)}}\Phi\left(\frac{|\varphi(x)-\varphi(y)|}{|x-y|}\right)|dx||dy|.
\end{align*}
Now suppose that $\varphi$ satisfies the so-called $\Phi$-Douglas condition, by monotone convergence theorem, we obtain
\begin{align*}
\int_{\partial\mathbb{D}}\int_{|x-y|\le 1}&\Phi\left(\frac{|\varphi(x)-\varphi(y)|}{|x-y|}\right)|dx||dy|=  \sum_{n=1}^\infty\int_{\partial\mathbb{D}}\int_{2^{-n}<|x-y|\le 2^{-(n-1)}}\Phi\left(\frac{|\varphi(x)-\varphi(y)|}{|x-y|}\right)|dx||dy|\\
&=\sum_{n=1}^\infty\sum_{k=1}^{2^n}\int_{I_{n,k}}\int_{2^{-n}<|x-y|\le 2^{-(n-1)}}\Phi\left(\frac{|\varphi(x)-\varphi(y)|}{|x-y|}\right)|dx||dy|.
\end{align*}
Let $I_{n+1,2k-1}$, $I_{n+1,2k}$ be the two dyadic ``children'' of $I_{n,k}$. For every $y\in I_{n,k} $, split $\{ x: 2^{-n}<|x-y|\le 2^{-(n-1)}\}$ into two parts,
\begin{align*}
  &I_{n,+}:=\left\{x:\frac{1}{2^n}<|x-y|\le \frac{1}{2^{n-1}},\ x\ {\rm in\ the\ counterclockwise\ direction\ from }\ y \right\}, \\
  &I_{n,-}:=\left\{x:\frac{1}{2^n}<|x-y|\le \frac{1}{2^{n-1}},\ x\ {\rm in\ the\ clockwise\ direction\ from }\ y \right\}.
\end{align*}
Then
\begin{align*}
\int_{I_{n,k}}&\int_{2^{-n}<|x-y|\le 2^{-(n-1)}}\Phi\left(\frac{|\varphi(x)-\varphi(y)|}{|x-y|}\right)|dx||dy|\\
&\ge\left(\int_{I_{n+1,2k-1}}\int_{I_{n,+}}+\int_{I_{n+1,2k}}\int_{I_{n,-}}\right)\Phi\left(\frac{|\varphi(x)-\varphi(y)|}{|x-y|}\right)|dx||dy|\\
&\ge\Big[\Phi\left(|\varphi(I_{n+1,2k-1})|2^{n-1}\right)+\Phi\left(|\varphi(I_{n+1,2k})|2^{n-1}\right)\Big]\cdot2^{-(2n+1)}.
\end{align*}

Using the identity
\begin{align*}
|\varphi(I_{n+1,2k-1})|+|\varphi(I_{n+1,2k})|=|\varphi(I_{n,k})|,\end{align*}
together with the convexity of $\Phi$ and the doubling condition, namely 
\begin{equation*}
    \Phi(a)+\Phi(b)\ge 2\Phi\left(\frac{a+b}{2}\right)\ge \frac{2}{C_\Phi}\Phi(a+b),
\end{equation*} 
we conclude that
\begin{equation}
\begin{aligned}
\int_{\partial\mathbb{D}}\int_{|x-y|\le 1}&\Phi\left(\frac{|\varphi(x)-\varphi(y)|}{|x-y|}\right)|dx||dy|\\
 &=\sum_{n=1}^\infty\sum_{k=1}^{2^n}
 \int_{I_{n,k}}\int_{2^{-n}<|x-y|\le 2^{-(n-1)}}\Phi\left(\frac{|\varphi(x)-\varphi(y)|}{|x-y|}\right)|dx||dy|\\
&\gtrsim\sum_{n=1}^\infty\sum_{k=1}^{2^n}\Phi(|\varphi(I_{n,k})|2^n)2^{-2n}.
\end{aligned}
\end{equation}

On the other hand, suppose that the double summation is finite. As the proof above, we have 
 \begin{align*}
\int_{\partial\mathbb{D}}\int_{|x-y|\le 1}&\Phi\left(\frac{|\varphi(x)-\varphi(y)|}{|x-y|}\right)|dx||dy|\\ 
&=\sum_{n=1}^\infty\sum_{k=1}^{2^n}\int_{I_{n,k}}\int_{2^{-n}<|x-y|\le 2^{-(n-1)}}\Phi\left(\frac{|\varphi(x)-\varphi(y)|}{|x-y|}\right)|dx||dy|.
\end{align*}
Note that if $y\in I_{n,k}$ and $|x-y|\le 2^{-(n-1)}$, then $|\varphi(x)-\varphi(y)|\le \sum_{j=k-2}^{k+2}|\varphi(I_{n,j})|$. Consequently,
\begin{align*}
   \int_{I_{n,k}}\int_{2^{-n}<|x-y|\le 2^{-(n-1)}}\Phi\left(\frac{|\varphi(x)-\varphi(y)|}{|x-y|}\right)|dx||dy|\le \Phi\left(\sum_{j=k-2}^{k+2}|\varphi(I_{n,j})|2^n\right)2^{-2n+1}.
\end{align*}
Noting that $I_{n,j}=I_{n,2^n+j}$ for $j=-1,0,...,3$, and using the doubling condition of $\Phi$, we obtain
\begin{align*}
   \sum_{k=1}^{2n}\Phi\left(\sum_{j=k-2}^{k+2}|\varphi(I_{n,j})|\right)
   \lesssim\sum_{k=1}^{2n}\sum_{j=k-2}^{k+2}\Phi\left(|\varphi(I_{n,j})|\right)
   =5\Phi\left(|\varphi(I_{n,k})|\right),
\end{align*}
where in the last step we used that each interval $I_{n,k}$ appears at most five times in the double sum. Consequently, 

\begin{align*}
    \int_{\partial\mathbb{D}}\int_{|x-y|\le 1}&\Phi\left(\frac{|\varphi(x)-\varphi(y)|}{|x-y|}\right)|dx||dy|\le \sum_{n=1}^\infty\sum_{k=1}^{2^n}\Phi\left(|\varphi(I_{n,k})|\right)2^{-2n},
\end{align*}
which proves the lemma.
\end{proof}
\section{Proof of Theorem \ref{thm:rkcp>2} and Corollaries}
\begin{proof}[Proof of Theorem \ref{thm:rkcp>2}] For the purposes of easier presentation we consider the boundary of the unit disk as locally flat, and hence we suppose that $\varphi$ is defined as a map of the interval $I:=[0,1]$ on the real line to itself. 
We follow the same line of arguments as in \cite[Proof of Theorem 1.5]{koski2023bi}, and the extension $h$ will be constructed in the upper half plane. We provide the details here for reader's convenience.

Step 1: For $n=1,\cdots$ we let $I^{n}:=I\times \{2^{-n}\}$ and $J^{n}:= I\times \{2^{-n}\} $ denote the unit length line segments obtained by lifting the interval $I$ to height $2^{-n}$ on the domain and target side, respectively. For each $n$ we also lift the boundary map
$\varphi$ to a map $\varphi^n$ which maps $I^n$ to $J^n$. Let us start by decomposing $I^n$ dyadically into segments $I^n_k$, $k=1,\cdots, 2^n$ of length $2^{-n}$, and the corresponding segments $J^n_k$ on the target side are defined by $J^n_k:=\varphi^n(I^n_k)$.

\begin{figure}[H]
\centering
\tikzset{every picture/.style={line width=0.75pt}} 

\begin{tikzpicture}[x=0.75pt,y=0.75pt,yscale=-1,xscale=1]

\draw   (70.33,119.33) -- (230.02,119.33) -- (230.02,180) -- (70.33,180) -- cycle ;
\draw    (69.33,200.33) -- (231,200) ;
\draw [shift={(150.17,200.17)}, rotate = 179.88] [color={rgb, 255:red, 0; green, 0; blue, 0 }  ][line width=0.75]    (0,5.59) -- (0,-5.59)   ;
\draw   (349.67,119.22) -- (509,119.22) -- (509,181.89) -- (349.67,181.89) -- cycle ;
\draw    (430,200.33) -- (510,200.33) ;
\draw [shift={(470,200.33)}, rotate = 180] [color={rgb, 255:red, 0; green, 0; blue, 0 }  ][line width=0.75]    (0,5.59) -- (0,-5.59)   ;
\draw [color={rgb, 255:red, 208; green, 2; blue, 27 }  ,draw opacity=1 ]   (69.33,159.33) -- (229.47,160) ;
\draw [color={rgb, 255:red, 208; green, 2; blue, 27 }  ,draw opacity=1 ]   (349.47,160) -- (431,160.33) ;
\draw    (349.67,200.33) -- (430,200.33) ;
\draw [color={rgb, 255:red, 208; green, 2; blue, 27 }  ,draw opacity=1 ]   (431,160.33) -- (509,160.33) ;
\draw [shift={(470,160.33)}, rotate = 180] [color={rgb, 255:red, 208; green, 2; blue, 27 }  ,draw opacity=1 ][line width=0.75]    (0,5.59) -- (0,-5.59)   ;
\draw    (150.33,118.33) -- (150.33,179.33) ;
\draw    (110.33,159.33) -- (110.33,179.33) ;
\draw    (190.33,159.33) -- (190.33,179.33) ;
\draw    (470,119.22) -- (470,181.33) ;
\draw    (490,160.33) -- (490,182.33) ;
\draw    (380,160.33) -- (380,182.33) ;
\draw    (268.89,161) -- (316,161.21) ;
\draw [shift={(318,161.22)}, rotate = 180.26] [color={rgb, 255:red, 0; green, 0; blue, 0 }  ][line width=0.75]    (10.93,-3.29) .. controls (6.95,-1.4) and (3.31,-0.3) .. (0,0) .. controls (3.31,0.3) and (6.95,1.4) .. (10.93,3.29)   ;
\draw    (268.89,200) -- (316,200.21) ;
\draw [shift={(318,200.22)}, rotate = 180.26] [color={rgb, 255:red, 0; green, 0; blue, 0 }  ][line width=0.75]    (10.93,-3.29) .. controls (6.95,-1.4) and (3.31,-0.3) .. (0,0) .. controls (3.31,0.3) and (6.95,1.4) .. (10.93,3.29)   ;
\draw    (269.89,131) -- (317,131.21) ;
\draw [shift={(319,131.22)}, rotate = 180.26] [color={rgb, 255:red, 0; green, 0; blue, 0 }  ][line width=0.75]    (10.93,-3.29) .. controls (6.95,-1.4) and (3.31,-0.3) .. (0,0) .. controls (3.31,0.3) and (6.95,1.4) .. (10.93,3.29)   ;

\draw (243,148) node [anchor=north west][inner sep=0.75pt]  [font=\small] [align=left] {$\displaystyle I^{n}$};
\draw (326,147) node [anchor=north west][inner sep=0.75pt]  [font=\small] [align=left] {$\displaystyle J^{n}$};
\draw (286,112) node [anchor=north west][inner sep=0.75pt]  [font=\small] [align=left] {$\displaystyle h$};
\draw (285,141) node [anchor=north west][inner sep=0.75pt]  [font=\small] [align=left] {$\displaystyle \varphi ^{n}$};
\draw (283,179) node [anchor=north west][inner sep=0.75pt]  [font=\small] [align=left] {$\displaystyle \varphi $};
\draw (97,138) node [anchor=north west][inner sep=0.75pt]  [font=\small] [align=left] {$\displaystyle \mathnormal{I_{k}^{n}}$};
\draw (399,138) node [anchor=north west][inner sep=0.75pt]  [font=\small] [align=left] {$\displaystyle \mathnormal{J_{k}^{n}}$};

\end{tikzpicture}
\caption{Construction of $h$}
\end{figure}
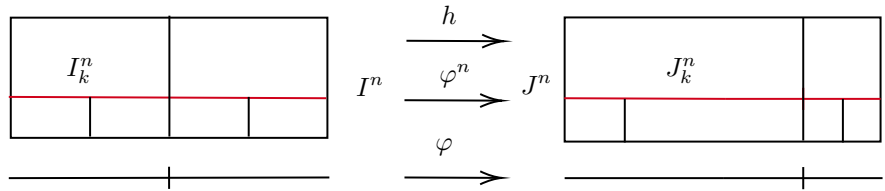

Step 2: We define a collection of new segments $S^n_j$ inductively as follows. We start by choosing $k_1$ as the smallest positive integer so that the union $J^n_1\cup\dots\cup J_{k_1}^n$
has length at least $2^{-n}$, and define $S_1^n, T_1^n$ as the unions
\begin{align*}
    S_1^n=I_1^n\cup\cdots\cup I_{k_1}^n \quad{\rm and}\quad T_1^n=J_1^n\cup\cdots\cup J_{k_1}^n,
\end{align*}
respectively. Then we continue this process by defining $k_j$ as the smallest positive integer so that the union $J^n_{k_{j-1}+1}\cup\cdots\cup J^n_{k_j}$ has length at least $2^{-n}$, and define
\begin{align*}
    S^n_j=I^n_{k_{j-1}+1}\cup\cdots\cup I^n_{k_j}\quad{\rm and}\quad T^n_j=J^n_{k_{j-1}+1}\cup\cdots\cup J^n_{k_j}.
\end{align*}
We omit the technical details shown in \cite[Proof of Theorem 1.5]{koski2023bi}, which guarantees
that $I^n$ can been decomposed into a number of segments $S^n_j$ of length at least $2^{-n}/2$, each of whose ``image" segment $T^n_j$ has length at least $2^{-n}$. The map $h$ is now finally defined on $I^n$ so that it maps each $S^n_j$ to $T^n_j$ linearly.

Step 3: We then extend $h$ to the strip between $I^n$ and $I^{n+1}$ for each $n$. Let $I^*$ be one of the segments $I^n_k$ in $I^n$ and let $I^*_-$ and $I^*_+$ denote the two intervals in $I^{n+1}$ that have half the length of $I^*$ and lie directly below it. Let $X_1$, $X_2$ denote the left and right endpoints of $I^*$ and $Y_1,Y_2,Y_3$ denote the endpoints of $I^*_-$ and $I^*_+$  from left to right with $Y_2$ being the common endpoint between. We now triangulate the quadrilateral $X_1Y_1Y_3X_2$ via the segments $X_1Y_1, X_1Y_2, X_2Y_2$ and $X_2Y_3$ into three triangles $\triangle_i$, $i=1,2,3$, on each of which we define the map $h$ as a linear map into the corresponding image triangle $\triangle^{'}_i$.

\begin{figure}
\centering
\tikzset{every picture/.style={line width=0.75pt}} 

\begin{tikzpicture}[x=0.75pt,y=0.75pt,yscale=-1,xscale=1]

\draw   (129.8,90.98) -- (249.42,90.98) -- (249.42,149.79) -- (129.8,149.79) -- cycle ;
\draw   (350.22,91.2) -- (469.69,91.2) -- (469.69,149.38) -- (350.22,149.38) -- cycle ;
\draw    (190.48,145) -- (190.48,154.6) ;
\draw    (130.8,90.98) -- (190.83,149.36) ;
\draw    (249.42,90.98) -- (190.83,149.36) ;
\draw    (380.21,144.88) -- (380.21,156.08) ;
\draw    (350.22,91.2) -- (380.21,149.68) ;
\draw    (469.69,91.2) -- (380.21,149.68) ;
\draw   (142.35,121.4) -- (148.47,131.4) -- (136.47,131.4) -- cycle ;
\draw   (188.35,106.4) -- (194.47,116.4) -- (182.47,116.4) -- cycle ;
\draw   (229.35,124.4) -- (235.47,134.4) -- (223.47,134.4) -- cycle ;
\draw   (386.35,108.4) -- (392.47,118.4) -- (380.47,118.4) -- cycle ;
\draw  [dash pattern={on 4.5pt off 4.5pt}]  (214,108) -- (372.47,108.66) ;
\draw [shift={(374.47,108.67)}, rotate = 180.24] [color={rgb, 255:red, 0; green, 0; blue, 0 }  ][line width=0.75]    (10.93,-4.9) .. controls (6.95,-2.3) and (3.31,-0.67) .. (0,0) .. controls (3.31,0.67) and (6.95,2.3) .. (10.93,4.9)   ;

\draw (186,83.09) node [anchor=north west][inner sep=0.75pt]   [align=left] {$\displaystyle I^{*}$};
\draw (158,140.99) node [anchor=north west][inner sep=0.75pt]   [align=left] {$\displaystyle I_{-}^{*}$};
\draw (217.2,140.99) node [anchor=north west][inner sep=0.75pt]   [align=left] {$\displaystyle I_{+}^{*}$};
\draw (111.6,85.09) node [anchor=north west][inner sep=0.75pt]  [font=\footnotesize] [align=left] {$\displaystyle X_{1}$};
\draw (251.6,84.29) node [anchor=north west][inner sep=0.75pt]  [font=\footnotesize] [align=left] {$\displaystyle X_{2}$};
\draw (184.4,153.89) node [anchor=north west][inner sep=0.75pt]  [font=\footnotesize] [align=left] {$\displaystyle Y_{2}$};
\draw (112.4,150.69) node [anchor=north west][inner sep=0.75pt]  [font=\footnotesize] [align=left] {$\displaystyle Y_{1}$};
\draw (251.42,152.79) node [anchor=north west][inner sep=0.75pt]  [font=\footnotesize] [align=left] {$\displaystyle Y_{3}$};
\draw (332.8,150.69) node [anchor=north west][inner sep=0.75pt]  [font=\footnotesize] [align=left] {$\displaystyle Y_{1}^{'}$};
\draw (331.6,81.89) node [anchor=north west][inner sep=0.75pt]  [font=\footnotesize] [align=left] {$\displaystyle X_{1}^{'}$};
\draw (318.8,215.89) node [anchor=north west][inner sep=0.75pt]   [align=left] {$ $};
\draw (406,80.69) node [anchor=north west][inner sep=0.75pt]  [font=\normalsize] [align=left] {$\displaystyle I^{*'}$};
\draw (473.2,82.69) node [anchor=north west][inner sep=0.75pt]  [font=\footnotesize] [align=left] {$\displaystyle X_{2}^{'}$};
\draw (373.4,155.89) node [anchor=north west][inner sep=0.75pt]  [font=\footnotesize] [align=left] {$\displaystyle Y_{2}^{'}$};
\draw (471.69,152.38) node [anchor=north west][inner sep=0.75pt]  [font=\footnotesize] [align=left] {$\displaystyle Y_{3}^{'}$};
\draw (355,139.09) node [anchor=north west][inner sep=0.75pt]   [align=left] {$\displaystyle I_{-}^{*'}$};
\draw (435.2,139.09) node [anchor=north west][inner sep=0.75pt]   [align=left] {$\displaystyle I_{+}^{*'}$};
\draw (122,144) node [anchor=north west][inner sep=0.75pt]   [align=left] {$ $};
\draw (148.35,126.4) node [anchor=north west][inner sep=0.75pt]  [font=\scriptsize] [align=left] {1};
\draw (196.35,111.4) node [anchor=north west][inner sep=0.75pt]  [font=\scriptsize] [align=left] {2};
\draw (237.35,129.4) node [anchor=north west][inner sep=0.75pt]  [font=\scriptsize] [align=left] {3};
\draw (393.35,114.4) node [anchor=north west][inner sep=0.75pt]  [font=\scriptsize] [align=left] {2};
\draw (390,103) node [anchor=north west][inner sep=0.75pt]   [align=left] {'};
\draw (290,89) node [anchor=north west][inner sep=0.75pt]   [align=left] {$\displaystyle h$};

\end{tikzpicture}
\caption{The piecewise linear construct of $h$}
\end{figure}
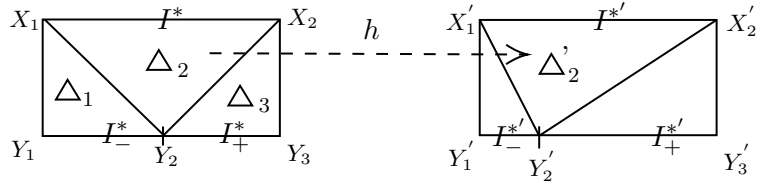

We first prove the ``if'' part. Note that the 
image triangle $\triangle^{'}_i$ has height $2^{-n}$, has a horizontal side whose length is $h(I')$, where $I'$ denotes one of the intervals $I^*, I^{*}_-$ or $ I^{*}_{+}$ depending on choice of $i$. Since the map $h$ as
a linear map from $\triangle_i$ into the corresponding image triangle $\triangle'_i$, we can estimate the differential and distortion of a linear map between two corresponding triangles: one with a base of length approximately $2^{-n}$ and a height of $2^{-n}$, another with a base of length $h(I')$ and a height of $2^{-n}$. Then 
\begin{align*}
&\int_{\triangle_i}\Phi(|Dh|)dz\lesssim\Phi(|h(I')|2^{n})2^{-2n}+\Phi(1)2^{-2n}\quad{\rm and}
\\
&\int_{\triangle'_i}\Psi(|Dh^{-1}|)dz
\lesssim \Psi\left(\frac{2^{-n}}{|h(I')|}\right)|h(I')|2^{-n}+\Psi(1)|h(I')|2^{-n}.
\end{align*}
where we omit the constant $C$, which depends only on $C_\Phi$ and $C_\Psi$.

Let us begin by bounding $\int_\mathbb{D}\Phi(|Dh|)dz$. Note that the term $\Phi(|h(I')|)2^{n}$ is only relevant when $|h(I')|$ is much larger than $2^{-n}$. In this case the interval $h(I')$ was obtained from $\varphi(I')$ in the construction either by
changing nothing or removing a small segment. In either case $|h(I')|\le |\varphi(I')|$. Now
\begin{align*}
\int_{\mathbb{D}}\Phi(|Dh|)\lesssim\sum_{n=1}^\infty\sum_{k=1}^{2^n}\Phi(|\varphi(I_{n,k})|2^n)2^{-2n}+
\Phi(1)<\infty,
\end{align*}
which is known to be finite since $\varphi$ satisfies the $\Phi$-Douglas condition and thus also the discrete $\Phi$-Douglas condition.

Let us then bound $\int_{\mathbb{D}}\Psi(|Dh^{-1}|)dz$. Since $\sum_{n=1}^\infty\sum_{k=1}^{2^n}\Psi(1)h(I')2^{-n}\le \Psi(1)$, it is enough to bound the first term $\Psi\left(\frac{2^{-n}}{|h(I')|}\right)|h(I')|2^{-n}$. Suppose that the segment $I'$ is part of the segment $S_j^{(n)}$, and that $S_j^{(n)}$ consists of the union of $N$ neighbouring segments of the same length. We estimate the total energy coming from all these segments, which equals
\begin{align}\label{N segments energy}
    N\cdot \Psi\left(\frac{2^{-n}}{|h(I')|}\right)|h(I')|2^{-n}.
\end{align}

Note that the map $h$ is linear on $S^n_j$ and  thus maps each of the $N$ segments into a segment of the same length $|h(I')|$, $2^{-n}\le|T^n_j|=N|h(I')|\le 2^{-n+1}$. But since the image segments $T_j^{n}=h(S_j^{n})$ was chosen to have length larger than $2^{-n}$ but less than $2^{-n+1}$, $T_j^{(n)}$ can be covered by at most three neighbouring dyadic segments $J_l^{(n)}$, we have $|S_j^{n}|\le 3 \max_l|\varphi^{-1}(J_l^{n})|$. Thus, from the energy \eqref{N segments energy},  we get
\begin{align*}
   N\cdot \Psi\left(\frac{2^{-n}}{|h(I')|}\right)2^{-n}|h(I')|
   &= N\cdot \Psi\left(\frac{|S^{n}_j|/N}{|T^{(n)}_j|/N}\right)2^{-n}|h(I')|\\
   &\le \Psi\left(2^n\max_l|\varphi^{-1}(J_l^{(n)})| \right) 2^{-2n+1}.
\end{align*}
Considering that each dyadic interval $J_l^{(n)}$ is involved in this process at most three times, we may sum over $n$ and $l$ to obtain that
\begin{align*}
    \int_{\mathbb{D}}\Psi(|Dh^{-1}|)dz\le \sum_{n=1}^{\infty}\sum_{l=1}^{2^n}\Psi\left(|\varphi^{-1}(J_l^{(n)})|2^n \right) 2^{-2n},
\end{align*}
where the double summation is finite due to the discrete  $\Psi$-Douglas condition for $\varphi^{-1}$. This completes the proof.

 We then prove the ``only if'' part. For every $n\ge1 $ and $1\le k\le 2^n$, let $Q_{n,k}$ denote the square of side length $2^{-n}$, whose  base side coincides with the lift of $I_{n,k}$ by 
$2^{-n}$. Fix such a square $Q_{n,k}$ and $I_{n,k}$, we denote by $\hat{\xi}$ the midpoint of $I_{n,k}$, and define $\gamma_t=\partial B(\hat{\xi},t)\cap\mathbb{D}$ for 
$2^{-(n+1)}\le t\le 3\cdot2^{-(n+1)}$. Then there exists an absolute constant $C$ such that
\begin{equation*}
    \bigcup_{t\in[1/2^{-(n+1)},3\cdot 2^{-(n+1)}]}\gamma_t\subset CQ_{n,k}\cap\mathbb{D}.
\end{equation*}

Since $h\in W^{1,1}(\mathbb{D},\mathbb{D})$ is a homeomorphism on $\overline{\mathbb{D}}$, $h$ is differentiable a.e. in $\mathbb{D}$. Now
\begin{equation*}
    |\varphi(I_{n,k})|\le\int_{\gamma_t}|Dh|ds
\end{equation*}
for $2^{-(n+1)}\le t\le 3\times2^{-(n+1)}$. By integrating with respect to $t$ we obtain
\begin{equation}\label{max-1}
    2^{-n}|\varphi(I_{n,k})|\le \int_{2^{-(n+1)}}^{3\cdot 2^{-(n+1)}}\int_{\gamma_t}|Dh|dsdt\le \int_{CQ_{n,k}\cap\mathbb{D}}|Dh(x)|dx.
\end{equation}
Note that $|CQ_{n,k}\cap\mathbb{D}|$ is uniformly comparable to $|Q_{n,k}|\approx 2^{-2n}$. It follows from \eqref{max-1} that
\begin{equation}\label{max-2}
2^n|\varphi(I_{n,k})|\lesssim \fint_{CQ_{n,k}\cap\mathbb{D}}|Dh(y)|dy\le M|Dh(x)|,\quad x\in Q_{n,k},
\end{equation}
where $M$ is the Hardy-Littlewood operator. For each $Q_{n,k}$, integrating both sides of \eqref{max-2} over $Q_{n,k}$ and applying the Lemma \ref{maximum embedding}, we obtain
\begin{equation*}\
\begin{aligned}
2^{-2n}\Phi(2^n|\varphi(I_{n,k})|)&\lesssim\int_{Q_{n,k}}\Phi(M|Dh(x)|)dx\\
&\lesssim\int_{Q_{n,k}}\Phi(|Dh(x)|)dx+|Q_{n,k}|\Phi(t_0).
\end{aligned}
\end{equation*}
By summing over all $n$ and $k$, it follows that
\begin{equation*}
\begin{aligned}
\sum_{n=1}^\infty\sum_{k=1}^{2^n}2^{-2n}\Phi(2^n|\varphi(I_{n,k})|)&\lesssim\sum_{n=1}^\infty\sum_{k=1}^{2^n}\int_{Q_{n,k}}\Phi(|Dh(x)|)dx+\sum_{n=1}^\infty\sum_{k=1}^{2^n}|Q_{n,k}|\Phi(t_0)\\
&=\int_{\mathbb{D}}\Phi(|Dh(x)|)dx+|\mathbb{D}|\Phi(t_0)<\infty.
\end{aligned}
\end{equation*}
Thus, by Lemma \ref{alternative formulation lemma}, we conclude that the boundary homeomorphism $\varphi:\partial\mathbb{D}\xrightarrow{\rm onto}\partial\mathbb{D}$ satisfies the $\Phi$-Douglas condition. A similar argument holds for $h^{-1}$, which completes the proof.
\end{proof}

Since Corollary \ref{cor: direct} is an immediate consequence of Theorem \ref{thm:rkcp>2}, we restrict ourselves to proving Corollary \ref{cor: phi},

\begin{proof}[Proof of Corollary \ref{cor: phi}]
Since $\Phi$ is an $N$-function satisfying doubling condition, by Lemma \ref{inverse convex}, we have
\begin{align*}
\sum_{n=1}^\infty\sum_{k=1}^{2^n}\Phi(|\varphi(I_{n,k})|2^n)2^{-2n}\le \sum_{n=1}^\infty\Phi(2^n)2^{-2n},
\end{align*}
where we used the fact that  $\sum_{k=1}^{2^n}|\varphi(I_{n,k})|=1$ for every $n$. Note that \eqref{eq:p<2orlicz} possesses a discrete counterpart,
\begin{align*}
\sum_{n=1}^\infty \Phi(2^n)2^{-2n}\approx\sum_{n=0}^\infty\int_{2^n}^{2^{n+1}}\frac{\Phi(t)}{t^3}dt=\int_1^\infty\frac{\Phi(t)}{t^3}dt<\infty.
\end{align*}
Hence, $\varphi$ satisfies the $\Phi$-Douglas condition. By Theorem \ref{thm:rkcp>2} and Remark \ref{rem only if part}, it admits a homeomorphic extension $h \colon \overline{\mathbb{D}} \xrightarrow{\rm onto}  \overline{\mathbb{D}}$ such that $h$ belongs to the Orlicz--Sobolev space $W^{1,\Phi}(\mathbb{D},\mathbb{C})$. A similar argument holds for $\varphi^{-1}$, which completes the proof.
\end{proof}
\newpage

 \bibliographystyle{alpha}
\bibliography{socalleddouglas}
\end{document}